\documentclass[12pt]{article}
\usepackage{amsmath,amsthm,amssymb}
\allowdisplaybreaks
\usepackage{color}

\newcommand{\ds}{\displaystyle}

\newcommand{\pa}{\partial}

\newcommand{\la}{\lambda}

\newcommand{\va}{\varepsilon}

\newcommand{\dis}{\displaystyle}

\theoremstyle{plain}
\newtheorem{theorem}{Theorem}[section]

\newtheorem{lemma}{Lemma}[section]

\numberwithin{equation}{section}
\newcommand{\n}{{\mathcal  N}}
\newcommand{\R}{{\mathbb R}}
\newcommand{\N}{{\mathbb N}}

\begin{document}
	\title{A Global Compact Result for a Fractional Elliptic Problem with  Hardy Term and  Critical Sobolev Non-linearity on the Whole Space
		}
	\author{ Lingyu Jin}
	
	\date{}
	\maketitle
	\begin{abstract}
		In this paper, we deal with a fractional elliptic equation with critical Sobolev nonlinearity and Hardy term
		$$\begin{cases}
		(-\Delta)^{\alpha} u-\mu\frac{u}{|x|^{2\alpha}}+a(x) u=|u|^{2^*-2}u+k(x)|u|^{q-2}u,\ \
		\,\\
		u\,\in\,H^\alpha(\R^N),
		\end{cases}
		\eqno {(*)} $$
		where   $2<q< 2^*$, $0<\alpha<1$, $N>4\alpha$,
		$2^*=2N/(N-2\alpha)$ is the critical Sobolev exponent, $a(x),k(x)\in C(\R^N)$. Through a compactness
		analysis of the functional associated to $(*)$, we
		obtain the existence of positive solutions for $(*)$ under
		certain assumptions on $a(x),k(x)$.
	\end{abstract}
	
	\bigbreak
	\bigbreak
	\bigbreak

	\section{Introduction}
	We consider the following nonlinear elliptic equation:
	\begin{equation}\label{1.1}
	\begin{cases}
	(-\Delta)^{\alpha} u-\mu\frac{u}{|x|^{2\alpha}}+a(x) u=\dis{|u|^{2^*-2}u}+k(x)|u|^{q-2}u, \,x\in \R^N,\ \
	\,\\
	u\,\in\,H^\alpha(\R^N),
	\end{cases}
	\end{equation}
	where  $2<q <2^*$, $0<\alpha<1$, $N>4\alpha$,  $2^*=2N/(N-2\alpha)$ is the critical Sobolev exponent,  $a(x),k(x)\in C(\R^N)$.
	
Recently the fractional Laplacian and more general nonlocal operators of elliptic type have been widely studied, both for their interesting theoretical structure and concrete applications in many fields such as optimization, finance, phase transitions, stratified materials, anomalous diffusion and so on (see \cite{BV,DMP,FB,FQ,NPV,SV,SV1,SZY}). In particular, a lot of  results have been accumulated for elliptic equations with critical nonlinearity related to (1.1). In \cite{DMP}, Dipierro etc. considere'd the critical problem with Hardy-Leray potential
	\begin{equation}\label{1.1j}
	\begin{cases}
	(-\Delta)^{\alpha} u-\mu \frac{u}{|x|^{2\alpha}}=\dis|u|^{2^*-2}u, \,&x\in \R^N,\ \
	\,\\
	u\in \dot{H}^\alpha(\R^N).
	\end{cases}
	\end{equation}
	where  $\dot{H}^\alpha (\R^N)$ is defined in (\ref{1.5jj}).
	They  proved the existence, certain qualitative properties and asymptotic behavior of positive solutions to (\ref{1.1j}). Ghoussoub  and Shakerian in \cite{GS} investigated the following double critical problem in $\R^N$
	\begin{equation}\label{1.1*}
	\begin{cases}
	(-\Delta)^{\alpha} u-\mu \frac{u}{|x|^{2\alpha}}=\dis\frac{|u|^{2^*-2}u}{|x|^s}+|u|^{2^*-2}u, \,&x\in \R^N,\ \
	\,\\
	u>0,\, u\in \dot{H}^\alpha(\R^N),
	\end{cases}
	\end{equation}
 with $\mu>0,0<s<2$.
 Through the non-compactness analysis of the Palais-Smale sequence of (\ref{1.1*}),  the existence of the solutions were obtained.
  The authors in \cite{JD} established  a concentration-compactness result for a fractional Schr\"{o}dinger equation  with the subcritical nonlinearity $f(x,u)$. Motivated by \cite{DMP,GS,JD,JF,WY} we consider the existence of  positive solutions for problem (\ref{1.1}) in $\R^N$.    The main interest for this type of problems, in addition to the nonlocal fractional Laplacian is the presence of the singular potential $\frac{1}{|x|^{2\alpha}}$ related to the fractional Hardy's inequality. We recall the Hardy inequality(\cite{DMP}),
 	  \begin{equation}\label{1.4j}
 	  \Gamma_{N,\alpha}\big (\int_{\R^N}\frac{|u(x)|^{2}}{|x|^{2\alpha}}dx\bigl )\leq  c_{N,\alpha}\int_{\R^N}\frac{|u(x)-u(y)|^2}{|x-y|^{N+2\alpha}}dy, \forall  u\in C_0^\infty (\R^N),
 	  \end{equation}
 	  where \begin{equation}
 	  \label{1.8j}\Gamma_{N,\alpha}=2^{2\alpha}\frac{\Gamma^2(\frac{N+2\alpha}{4})}{\Gamma^2(\frac{N-2\alpha}{4})}, c_{N,\alpha}=2^{2\alpha-1}\pi^{-\frac{N}{2}}\frac{\Gamma(\frac{N+2\alpha}{2})}{|\Gamma(-\alpha)|}.\end{equation}
 	 The Sobolev embedding $\dot{H}^\alpha(\R^N) \hookrightarrow L^{2}(|x|^{-2\alpha}, \R^N)$ is not compact, even locally, in any neighborhood of zero.  As it is well known, the loss of the compactness of the  embeddings  is one of the  main difficulties for elliptic problems with critical nonlinearities. Problem (1.1) has three factors, critical Sobolev term, Hardy term and  unbounded domain which lead to the non-compactness of the embeddings. In \cite{DMP} and \cite{GS}, the authors can consider the solutions of critical problems in the homogeneous fractional Sobolev space $\dot{H}^\alpha(\R^N)$, while we must deal with (1.1) in the nonhomogeneous fractional Sobolev space $H^\alpha(\R^N)$ given the presence of low sub-critical terms in (\ref{1.1}). This is why the methods in \cite{DMP} and \cite{GS} can not be used directly to (\ref{1.1}). As far as we know, the existence results  for the fractional elliptic problems with a mixture of critical Sobolev terms, Hardy term and subcritical terms are relatively new. To overcome the difficulties caused by the lack of compactness, we carry out a non-compactness analysis which can distinctly express all the parts which cause non-compactness. As  a result, we are able to obtain the existence of nontrival solutions of the elliptic problem with the critical nonlinear term on an unbounded domain  by getting rid of these noncompact factors.
	To be more specific, for the Palais-Smale sequences of the variational functional
	corresponding to (\ref {1.1}) we first establish a  complete noncompact expression  which includes all the blowing up
	bubbles caused by the critical Sobolev nonlinearity, the Hardy term and by the unbounded
	domain. Then we derive the existence  of positive solutions for  (\ref {1.1}).  Our methods are based on   some techniques of \cite{DJP,JD,L1,PP,Sm,S2,ZC,Y}.

	Before introducing our main results, we give  some
	notations and assumptions.
	
	{\bf Notations and assumptions:}
	
	 Denote $c$ and $C$ as
	 arbitrary constants which may change from line to line. Let $B(x,r)$ denote a ball centered at $x$ with
	 radius $r$ and $B(x,r)^C=\R^N\setminus B(x,r)$.

	Let $N\geq 1$, $u\in L^2(\R^N)$, let the Fourier transform of $u$ be
	$$\widehat{u}(\xi)=\frac{1}{(2\pi)^{\frac{N}{2}}}\int_{\R^N}e^{-i\xi\cdot x}u(x)dx.$$
		We define the operator $(-\Delta)^\alpha u$ by the Fourier transform $$\widehat{(-\Delta)^{\alpha} u}(\xi)=|\xi|^{2\alpha}\hat{u}(\xi),\ \ \forall u\in C^\infty _0(\R^N).$$
	Let $\dot{H}^\alpha(\R^N)$ be the homogeneous fractional Sobolev space  as the completion of $C^\infty _0(\R^N)$ under the norm
	\begin{equation}\label{1.5jj}
	\|u\|_{\dot{H}^\alpha(\R^N)}=\||\xi|^\alpha\widehat{u}\|_{L^2(\R^N)},
	\end{equation}
	and  denote by $H^\alpha(\R^N)$ the usual nonhomogeneous fractional Sobolev space with the norm
	\begin{equation}\label{1.6}
	\|u\|_{H^\alpha(\R^N)}=\|u\|_{L^2(\R^N)}+\||\xi|^\alpha\widehat{u}\|_{L^2(\R^N)}.
	\end{equation}
 For $0<\alpha<1$, a direct calculation  (see e.g. [\cite{NPV}, proposition 4.4] or [\cite{DMP}, Proposition 1.2]) gives
$$c_{N,\alpha}\int_{\R^N}\int_{ \R^N} \frac{|u(x)-u(y)|^2}{|x-y|^{N+2\alpha}}dxdy=\int_{\R^N}|(-\Delta )^{\alpha/2}u(x)|^2dx=\|u\|_{\dot{H}^\alpha(\R^N)}^2,$$
where $c_{N,\alpha}=2^{2\alpha-1}\pi^{-\frac{N}{2}}\frac{\Gamma(\frac{N+2\alpha}{2})}{|\Gamma(-\alpha)|}$.

Let $u^+=\max \{u,0\}, u^-=u^+-u$. From the proof of (\ref{2.10}) in \cite{PAJ}, it follows
\begin{equation}
\|u^+\|_{\dot {H}^\alpha}\leq \|u\|_{\dot {H}^\alpha}.
\end{equation}

We call $u\not\equiv 0$ in $\R^N$ if the measure of the set $\{x\in \R^N|u(x)\not =0\}
$ is positive.

Recall the definition of Morrey space. A measurable function $u:\R^N \rightarrow \R$ belongs to the Morrey space  with $p\in [1,\infty)$ and $\nu \in (0,N]$, if and only if
\[ \|u\|^p_{L^{p,\nu}(\R^N)}=\sup_{r>0,\bar x\in \R^N}r^{\nu-N}\int_{B(\bar x,r)}|u(x)|^pdx<\infty. \]
By H\"older inequality, we can verify  (refer to \cite{NPV})
 \begin{equation}
 L^{2^*}(\R^N)\hookrightarrow L^{p,\nu}(\R^N),  \text{ for } 1\leq p<2^*,
 \end{equation}
and
 \begin{equation}
L^{p,\frac{(N-2\alpha)p}{2}}(\R^N)\hookrightarrow L^{p_1,\frac{(N-2\alpha)p_1}{2}}(\R^N),  \text{ for } 1<p_1<p<2^*.
 \end{equation}
 Moreover, we have $ L^{p,\nu}(\R^N)\hookrightarrow  L^{1,\frac{\nu}{p}}(\R^N)$.

 Next we give the definition of the Palais-Smale sequence.
 Let $X$ be a Banach space, $\Phi\in C^1(X,\R)$, $c\in \R$, we call $\{u_n\} \subset X$ is a Palais-Smale sequence of $\Phi$ if
 \begin{equation}
 	\Phi(u_n)\rightarrow c,\,\,\Phi'(u_n)\rightarrow 0\text{ as }n\rightarrow \infty.
 \end{equation}
In this paper we assume that:

(a) $ a(x)\in C(\R^N)$, $ k(x)\in C(\R^N)$;

(b) $\displaystyle\lim_{|x|\rightarrow\infty} a(x)=\bar
a>0,\lim_{|x|\rightarrow\infty} k(x)=
\bar k>0,\,\,\,\inf_{x\in\R^N}a(x)={\hat
a}>0,\,\inf_{x\in\R^N}k(x)={\hat
k}>0.$

 In the
following, we assume that $a(x),k(x)$ always satisfy (a) and (b).
The energy functional associated with (1.1) is  for all $ u\in H^{\alpha}(\R^N)$,
$$\begin{aligned}I(u)&=\frac{1}{2}\int_{\R^N}\Bigl (|(-\Delta)^{\alpha/2}u(x)|^2
-\mu\frac{u^2}{|x|^{2\alpha}}+a(x)|u(x)|^2\Bigl )
dx\\& \ \ \ \ \ -\dis\frac{1}{2^*}\int_{\R^N}(u^+(x))^{2^*}dx
-\frac{1}{q}\int_{\R^N}k(x)(u^+(x))^q
dx.\end{aligned}$$ Finally we present some problems
associated to (1.1) as  follows.

\noindent The limit equation of (1.1)  involving subcritical and critical terms is
\begin{equation}\label{(2.1)}
\begin{cases}
(-\Delta)^{\alpha} u+\bar a u=\bar{k}|u|^{q-2}u+|u|^{2^*-2}u,\\
u\in H^{\alpha}(\R^N),
\end{cases}
\end{equation}
and its corresponding variational functional is
$$\begin{aligned}I^\infty(u)&=\frac{1}{2}\int_{\R^N}\Bigl (|(-\Delta)^{\alpha/2}u(x)|^2+\bar a
|u(x)|^2\Bigl
)dx-\frac{1}{q}\int_{\R^N}\bar k
(u^+(x))^qdx\\&\ \ \ -\frac{1}{2^*}\int_{\R^N}{(u^+(x))^{2^*}}dx, \,\,\ \  u\in H^{\alpha}(\R^N).\end{aligned} $$ The limit equation of (1.1)
involving the Hardy term  and critical Sobolev nonlinearity is
\begin{equation}\label{1.12}
\begin{cases}
(-\Delta)^{\alpha} u-\mu\frac{u}{|x|^{2\alpha}}=\dis{|u|^{2^*-2}u},\\
u\in \dot{H}^\alpha(\R^N),
\end{cases}
\end{equation}
and the corresponding variational functional is
$$I_\mu(u)=\frac{1}{2}\int_{\R^N} (|(-\Delta)^{\alpha/2}u(x)|^2-\mu\frac{u^2}{|x|^{2\alpha}})dx-\frac{1}{2^*}\int_{\R^N}{(u^+(x))^{2^*}}dx,\,\,\ \ \
u\in \dot{H}^\alpha(\R^N).$$
The limit equation of (1.1)
involving  critical Sobolev nonlinearity is
\begin{equation}\label{1.4}
\begin{cases}
(-\Delta)^{\alpha} u=\dis{|u|^{2^*-2}u},\\
u\in \dot{H}^\alpha(\R^N),
\end{cases}
\end{equation}
and the corresponding variational functional is
$$I_0(u)=\frac{1}{2}\int_{\R^N} |(-\Delta)^{\alpha/2}u(x)|^2dx-\frac{1}{2^*}\int_{\R^N}{(u^+(x))^{2^*}}dx,\,\,\ \ \
u\in \dot{H}^\alpha(\R^N).$$
Define
\begin{equation}\label{1.15}\dis
\ds S_{\alpha,\mu}=\inf_{u\in H^\alpha(\R^N)\backslash \{0\}}  \frac{c_{N,\alpha}\int_{\R^{2N}}\dis\frac{|u(x)-u(y)|^2}{|x-y|^{N+2\alpha}}dxdy-\mu\int_{\R^N}\frac{|u(x)|^{2}}{|x|^{2\alpha}}dx}{(\int_{\R^N}|u(x)|^{2^*}dx)^{2/2^*}},
\end{equation}	
the Euler equation associated to (\ref{1.15}) is  (\ref{1.12}).
In particular it has been showed  	in Theorem 1.2 of \cite{DMP} that for any positive solution $U_\mu(x) \in H^\alpha(\R^N)$ of (\ref{1.12}), there exist two positive constants $c, C$ such that
\begin{equation}\label{1.17}
\frac{c}{\Bigl(|x|^{1-\eta_\mu}(1+|x|^{2\eta_\mu})\Bigl )^{\frac{N-2\alpha}{2}}}\leq U_\mu(x)\leq \frac{C}{\Bigl (|x|^{1-\eta_\mu}(1+|x|^{2\eta_\mu})\Bigl )^{\frac{N-2\alpha}{2}}}, \text{ in } \R^N\backslash \{0\}
\end{equation}
where \begin{equation}\label{1.19}
\eta_\mu =1-\frac{2\alpha_\mu}{N-2\alpha},
\end{equation}
and $\alpha_\mu\in (0,\frac{N-2\alpha}{2})$ is a suitable parameter whose explicit value will be determined as the unique solution to the following equation
\begin{equation}\label{1.20jjj}
\varphi_{s,N}(\alpha_\mu)=2^{2\alpha} \frac{\Gamma(\frac{\alpha_\mu+2\alpha}{2})\Gamma(\frac{N-\alpha_\mu}{2})}{\Gamma(\frac{N-\alpha_\mu-2\alpha}{2})\Gamma(\frac{\alpha_\mu}{2})}=\mu,
\end{equation}
and $\varphi_{\alpha,N}$ is strictly increasing. That is
\begin{equation}\nonumber
\alpha_\mu=\varphi_{\alpha,N}^{-1}(\mu).
\end{equation}

 All the positive solutions of (\ref{1.12}) are of the form
 \begin{equation}\label{1.5j}
U_\mu^\varepsilon(x):=\varepsilon^{\frac{2\alpha-N}{2}} U_\mu
(x/\varepsilon).
 \end{equation}
 In particular, for $\mu=0$, it follows that (refer to \cite{C} )
  \begin{equation}\label{1.5} U_0(x)= \frac{C}{1+|x|^{N-2\alpha}},
 \end{equation}
 where $C>0$ is a constant.
 These solutions $U_0^{\va,y}:=\va^{\frac{2\alpha-N}{2}}U_0(\frac{x-y}{\va})$ are also minimizers
 for the quotient
 $$
 S_{\alpha}=\inf_{u\in
 	\dot{H}^\alpha(\mathbb{R}^N) \backslash\{0\}}\frac{\dis\int_{\mathbb{R}^N}|(-\Delta)^{\alpha/2}
 	u(x)|^2dx}{\Bigl(\dis\int_{\mathbb{R}^N}{|u(x)|^{2^*}}dx\Bigl)^{2/2^*}}.$$ 
 Define
\begin{equation}\label{1.11j}
D_\mu=\int_{\mathbb{R}^N}\Bigl(\frac{1}{2}(|(-\Delta)^{\alpha/2}
U_\mu(x)|^2-\mu\frac{U_\mu(x)^2}{|x|^{2\alpha}})-\frac{1}{2^*}|U_\mu(x)|^{2^*}\Bigl)dx=\frac{\alpha}{N}S_{\alpha,\mu}^{\frac{N}{2\alpha}},\end{equation}
\begin{equation}\label{1.10j}
D_0=\int_{\mathbb{R}^N}\Bigl(\frac{1}{2}|(-\Delta)^{\alpha/2}
U_0|^2-\frac{1}{2^*}|U_0|^{2^*}\Bigl)dx=\frac{\alpha}{N}S_{\alpha}^{\frac{N}{2\alpha}},\end{equation}
\begin{equation}\n=\{u\in H^{\alpha}(\R^N)\setminus
\{0\}\,\,|\,\,\int_{\R^N}\Bigl(|(-\Delta)^{\alpha/2}u(x)|^2+\bar
a|u(x)|^2-\bar k(u^+(x))^q\Bigl) dx=0\},
\end{equation}
\begin{equation}J^\infty=\inf_{u\in \n}I^\infty (u).\end{equation}
It is known that $\n \neq\emptyset$ since problem (\ref{(2.1)}) has
at least one positive solution if $N> 2\alpha$ (see Theorem 1.3 in \cite{ZZX})
for $2<q<2^*$ and $\bar k>\lambda^*(\lambda^*>0$ is a positive constant definded in \cite{ZZX}).

The main result of our paper is as follows:
\begin{theorem}\label{a}
Suppose $a(x),\,k(x)$ satisfy (a) and (b), $\bar k>\lambda^*$, $ 2<q<2^*$, $0<\alpha<1$, $N>4\alpha$, $0<\mu <\phi_{\alpha,N}(\frac{N-4\alpha}{2})$. Assume that $\{u_n\}$ is a positive
Palais-Smale sequence of I at level $d\geq 0$, then there exist
sequences $\{y^k_n\}\subset\R^N\,(1\leq k\leq
l_1)$,
$\{R^j_n\}\subset\R^+\,(1\leq j\leq l_2$), $\{\bar R^i_n\}\subset\R^+,\,\,\,x_n^i\subset\R^N (1\leq i\leq l_3$)  and   $u_k\in H^{\alpha}(\R^N) \,(1\leq k\leq
l_1) (l_1,l_2,l_3\in \N)$, $ u\in H^{\alpha}(\R^N)$, such that up to a subsequence:

$\,\,\displaystyle
d=I(u)+\sum^{l_1}_{j=1}I^\infty(u_k)+l_2D_\mu+l_3D_0;$

\begin{equation}\label{b1}\,\,
\|u_n-u-\sum^{l_1}_{k=1}u_k(x-y^k_n)-\sum^{l_2}_{j=1}U^{R^j_n}-\sum^{l_3}_{i=1}U_0^{\bar R^i_n, x^i_n}\|_{H^{\alpha}(\R^N)}=o(1) \text{ as } n\rightarrow \infty,
\,\,\hspace{2.5cm}\end{equation}
where $u$ and $u_k\,(1\leq k\leq l_1)$ satisfy
$$I' (u)=0,\,\,{I^{\infty} }'(u_k)=0$$ and
 $$|y^k_n|\rightarrow \infty\,(1\leq k\leq l_1), \,R^j_n\rightarrow 0\,\,(1\leq j\leq
l_2),\,\,\bar R^i_n\rightarrow 0\,\,,|\frac{x_n^i}{\bar R_n^i}|\rightarrow \infty\,\,(1\leq i\leq
l_3),\text{ as
}n\rightarrow\infty.$$
 In particular, if
$u\not\equiv 0$, then $u$ is a weakly solution of (\ref{1.1}). Note
that the corresponding sum in (\ref{b1}) will be treated as zero if
$l_i=0\,(i=1,2,3).$
\end{theorem}
{\bf Remarks:}


 1)  Similar as Corollary 3.3 in \cite{Sm}, one can show that any Palais-Smale sequence for $I$ at a level
 which is not of the form $m_1D_0+m_2D_\mu+m_3 J^\infty$,
 $m_1,m_2\in\N\bigcup \{0\}$, gives rise to a non-trivial weak
 solution of equation (1.1).

 2) In our non-compactness analysis, we prove
 that the blowing up positive Palais-Smale sequences can bear exactly
 three kinds of bubbles. Up to harmless constants, they are either
 of the form
 $$U_\mu^{R_n}(x), \text{
 }|R_n|\rightarrow 0\text{ as } n\rightarrow
 \infty, $$
 or
 $$U_0^{\bar R_n,x_n}(x), \text{
 }|\bar R_n|\rightarrow 0, \frac{|x_n|}{\bar R_n}\rightarrow \infty\text{ as } n\rightarrow
 \infty,
 $$
or  $$u(x-y_n)\in H^{\alpha}(\R^N),\,\text{
}|y_n|\rightarrow\infty,\text{ as } n\rightarrow
 \infty,$$where $u$ is the solution of
(\ref{(2.1)}). For any Palais-Smale sequence ${u_n} $ for $I$, ruling out the above two bubbles yields  the existence of a non-trivial weak solution of equation (1.1).

Using the compactness results and the Mountain Pass Theorem \cite{BN}
 we prove
the following existence result.
\begin{theorem}\label{b}
Assume that $2<q<2^*$, $0<\alpha<1$, $N> 4\alpha$, $0<\mu <\phi_{\alpha,N}(\frac{N-4\alpha}{2})$. 
 If $a(x),k(x)$ satisfy (a), (b) and
\begin{equation}
\bar a\geq a(x), k(x)\geq \bar k>0,\,k(x)\not
\equiv \bar k
\end{equation}
 Then  (\ref{1.1}) has a nontrivial solution $u\in H^{\alpha}(\R^N)$ which
satisfies $$I(u)< \min\{\frac{\alpha}{N}S_{\alpha,\mu}^{\frac{N}{2\alpha}}, J^\infty\}.$$
\end{theorem}

This paper is organized as follows.  In Section 2, we prove Theorem \ref{a} by
carefully analyzing  the features of a positive Palais-Smale
sequence for $I$. Theorem 1.2 is proved in Section 3 by
applying Theorem \ref{a} and the Mountain Pass Theorem.
Finally we put some preliminaries in the last section as an appendix.

\section{ Non-compactness analysis}
In this section, we prove Theorem 1.1 by using the Concentration-Compactness
Principle and a delicate analysis of the Palais-Smale sequences of
$I$.
Firstly we give the following Lemmas.
\begin{lemma}\label{l6}
	Let $0<\alpha <N/2 $,  $\{u_n\}\subset {\dot{H}}^\alpha(\R^N)$ be a bounded sequence such that
	\begin{equation}\label{l6.1}
	\inf_{n\in \N}\int_{\R^N}{\bigl(u^+_n(x)\bigl)^{2^*}}dx\geq c>0.
	\end{equation}
	Then, up to subsequence, there exist two sequences $\{r_n\}\subset \R^+$ and $\{x_n\}  \subset \R^N $
	such that
	\begin{equation}\label{l6.2}
	\bar u_n \rightharpoonup w \text{ in } \dot{H}^\alpha(\R^N) \text { with } w \not\equiv 0, \end{equation}
		where \begin{equation}\label{2.3j}
		\begin{aligned}
		\bar u_n(x)=\begin{cases}
	r_n^{\frac{N-2\alpha}{2}}u_n(r_n x) &\text{ when }\frac{x_n}{r_n} \text{ is bounded, }\\
	r_n^{\frac{N-2\alpha}{2}}u_n(r_n x+x_n) &\text { when } |\frac{x_n}{r_n}|\rightarrow \infty.
		\end{cases}
		\end{aligned}
		\end{equation}
\end{lemma}
\begin{proof}

	By Theorem 1 in \cite{PP},
	\begin{eqnarray}\label{2.4jj}
	 (\int_{\R^N}{|u_n(x)|^{2^*}}dx\bigl )^{1/2^*}\leq C \| u_n\|^\theta_{\dot{H}^\alpha(\R^N)}\| u_n\|^{1-\theta}_{L^{2,N-2\alpha}(\R^N)},
	\end{eqnarray}
	where $\frac{N-2\alpha}{N}\leq \theta<1$.

	Then there exists a constant $c>0$ such that
	\begin{equation}\label{2.22j}
	\| u_n\|_{L^{2,N-2\alpha}(\R^N)}^2=\sup_{\bar x\in \R^N,\, R\in \R^+}R^{-2\alpha}\int_{B(\bar x,R)}|u_n(x)|^2dx\geq c>0.
	\end{equation}

	From (\ref{2.22j}), we may find $r_n>0$ and $x_n \in \R^N$ such that for $n$ large enough,
	\begin{equation}\label{2.23j}
	r_n^{-2\alpha}\int_{B(x_n,r_n)}|u_n(x)|^2dx\geq \| u_n\|_{L^{2,N-2\alpha}(\R^N)}^2-\frac{c}{2n}\geq c/2>0.
	\end{equation}
	Denote	 \begin{equation}\label{2.3}
	\begin{aligned}
	\bar u_n(x)=\begin{cases}
	r_n^{\frac{N-2\alpha}{2}}u_n(r_n x) &\text{ when }\frac{x_n}{r_n} \text{ is bounded, }\\
	r_n^{\frac{N-2\alpha}{2}}u_n(r_n x+x_n) &\text { when } |\frac{x_n}{r_n}|\rightarrow \infty.
	\end{cases}
	\end{aligned}
	\end{equation}
	Since $\{u_n\}$ is bounded in $\dot{H}^\alpha(\R^N)$, from the scaling and translation invariance of  $\dot{H}^\alpha(\R^N)$, 	then $\{\bar u_n\}$ is bounded in $\dot{H}^\alpha(\R^N)$, therefore, up to a subsequence  (still denoted by $\bar u_n$),
	$$\bar u_n \rightharpoonup w \text{ in } \dot{H}^\alpha(\R^N) \text{ and }\bar u_n \rightarrow w \text{ in } L^2_{\mathrm{loc}}(\R^N), \text{ as } n\rightarrow \infty.$$
	
	If $\frac{x_n}{r_n}$ is bounded, there exists a $\tilde{R}>1$ such that $B(\frac{x_n}{r_n},1)\subset B(0,\tilde{R})$, then
	\begin{eqnarray}\label{2.9}
	c/2<\int_{B(\frac{x_n}{r_n},1)}|\bar u_n(x)|^2dx\leq \int_{B(0,\tilde{R})}|\bar u_n(x)|^2dx\rightarrow \int_{B(0,\tilde{R})}|w(x)|^2dx.
	\end{eqnarray}
	If $|\frac{x_n}{r_n}|\rightarrow \infty$, then
		\begin{eqnarray}\label{2.10j}
		c/2<\int_{B(0,1)}|\bar u_n(x)|^2dx\leq \int_{B(0,\tilde{R})}|\bar u_n(x)|^2dx\rightarrow \int_{B(0,\tilde{R})}|w(x)|^2dx
		\end{eqnarray}
			where $\tilde{R}>1$.
		Obviously we have $w\not \equiv 0$.
		From (\ref{2.9}) and (\ref{2.10j}), Lemma \ref{l6} is complete.
\end{proof}
\begin{lemma}\label{3.7}Assume $N>4\alpha,2<q< 2^*, 0<\alpha<1$.	Let $\{v_n\}\subset H^{\alpha}(\R^N)$ be a Palais-Smale  sequence of $I$ at
	level $d_1$ and $v_n\rightharpoonup 0\text{ in } H^{\alpha}(\R^N)\text{
		as }n\rightarrow \infty$. If there exists two sequence $\{r_n\}\subset
	\R^+$ and $\{x_n\}\in \R^N$ $\text{ with } r_n\rightarrow 0, \, \frac{|x_n|}{r_n}\rightarrow \infty$ as $n\rightarrow \infty$   such
	that $\bar v_n(x):=r_n^{\frac{N-2\alpha}{2}}v_n (r_n x+x_n)$ converges weakly
	in $\dot{H}^\alpha(\R^N)$ and almost everywhere to some $v_0\in
	\dot{H}^\alpha(\R^N)\text{ as }n\rightarrow \infty$ with $v_0\not\equiv 0$, then
	 $v_0$ solves
	(\ref{1.4}), the sequence
	$z_n(x):=v_n(x)-v_0(\frac{x-x_n}{r_n})r_n^{\frac{2\alpha-N}{2}}\rightharpoonup 0\text{ in } H^{\alpha}(\R^N)$ and $z_n(x)$ is a
	Palais-Smale sequence of $I$ at level $d_1-I_0(v_0)$. \end{lemma}
\begin{proof}
	First, we prove that $v_0$ solves (\ref{1.4}) and
	$I(z_n)=I(v_n)-I_0(v_0)$. Fix a ball $B(0,r)$ and a test
	function $\phi\in C^\infty_0(B(0,r))$.
	Since
$$	\begin{aligned}
\bar v_n\rightharpoonup v_0 \text{  in } \dot{H}^\alpha(\R^N),
	\end{aligned}$$
 it implies
	\begin{equation}\label{j2.13}
	\begin{split}
	&\ \ \ \langle\phi, I'_0( v_0)\rangle+o(1)
	\\&	=\langle\phi, I'_0(\bar v_n)\rangle\\&=c_{N,\alpha}\int_{\R^N}\int_{\R^N}\frac{(\bar v_n(x)-\bar v_n(y))(\phi(x)-\phi(y))}{|x-y|^{N+2\alpha}}
	dxdy-\int_{\R^N}\bigl(\bar v^+_n(x)\bigl)^{{2^*}-1}\phi(x)
	dx\\
	&=c_{N,\alpha}\int_{\R^N}\int_{\R^N}\frac{(\bar v_n(x)-\bar v_n(y))(\phi(x)-\phi(y))}{|x-y|^{N+2\alpha}}
	dxdy-\mu\int_{\R^N}\frac{\bar v_n(x)\phi(x)}{|x+\frac{x_n}{r_n}|^{2\alpha}}dx-\int_{\R^N}\bigl(\bar v^+_n(x)\bigl)^{{2^*}-1}\phi(x)
	dx\\
	& \ \ \ +r_n^{2\alpha}\int_{\R^N} a({r_n x+x_n})\phi(x)\bar v_n(x)
	dx-r_n^{N-\frac{N-2\alpha}{2}q}\int_{\R^N}k(r_n x+x_n)\phi(x) (\bar v^+_n(x))^{q-1}dx+o(1)\\
	&=c_{N,\alpha}\int_{\R^N}\int_{\R^N}\frac{(v_n(x)-v_n(y))(\phi_n(x)-\phi_n(y))}{|x-y|^{N+2\alpha}}
	dxdy-\mu\int_{\R^N}\frac{ v_n(x)\phi_n(x)}{|x|^{2\alpha}}dx-\int_{\R^N}(v^+_n(x))^{{2^*}-1}\phi_n(x)
	dx\\
	& \ \ \ +\int_{\R^N} a(x)\phi_n(x) v_n(x)dx -\int_{\R^N}
	k(x)\phi_n(x)(v^+_n(x))^{q-1} dx+o(1) =o(1)\text{ as }n\rightarrow
	\infty,
	\end{split}
	\end{equation}
where $\phi_n=r_n^{-\frac{N-2\alpha}{2}}\phi(\frac{x-x_n}{r_n})$.
 The last equality in  (\ref{j2.13})  holds since
 $$\int_{\R^N}|\phi_n(x)|^2dx=r_n^{2\alpha}\int_{\R^N}|\phi(x)|^2dx=o(1),$$
	$$\|\phi\|_{\dot{H}^\alpha(\R^N)}=\|\phi_n\|_{\dot{H}^{\alpha}(\R^N)}=\|\phi_n\|_{{H}^{\alpha}(\R^N)}+o(1),\text{ as
	}n\rightarrow \infty.$$
	Thus $v_0$ is a nontrival critical point of $I_0$. By  Lemma \ref{l4.6}, (\ref{1.5}) and the fact $N>4\alpha$, it follows
	\begin{equation}\label{2.12}
	\int_{\R^N}|v_0(x)|^pdx\leq c	\int_{\R^N}\frac{1}{(1+|x|^{N-2\alpha})^p}dx\leq c,\,\, \forall \,\, p\geq 2
	\end{equation}
	which implies that $v_0\in L^2(\R^N)$.
	Let
	$$ z_n(x)=v_n(x)-r_n^{\frac{2\alpha-N}{2}}v_0(\frac{x-x_n}{r_n})\in
	H^{\alpha}(\R^N).$$

	From (\ref{2.12}),\,$v_0\in L^p(\R^N)$ for all $p\in [2,2^*)$. Then it follows
	\begin{equation}\label{x}
	\begin{split}
	\int_{\R^N}|v_0(\frac{x-x_n}{r_n})r_n^{\frac{2\alpha-N}{2}}|^pdx=r_n^{N-p\frac{(N-2\alpha)}{2}}\|v_0\|^p_{L^p(\R^N)}\rightarrow 0 ,\,\,\text{ as }n\rightarrow\infty,\text{ for all }{2\leq p< {2^*}},
	\end{split}
	\end{equation}
	Thus $z_n\rightharpoonup 0 \text{
		in }H^{\alpha}(\R^N)\text{ as }n\rightarrow \infty$. Now we prove
	that $\{z_n\}$ is a Palais-Smale  sequence of $I$ at level $d_1-I_0
	(v_0)$.
	By the Br\'{e}zis-Lieb Lemma and the weak convergence, similar to Lemma \ref{q} in the Appendix, we can prove that
	$$I(z_n)=I(v_n)-I_0(v_0),$$
	$$ \langle I'(z_n),\phi\rangle=o(1)$$ as $n\rightarrow \infty$.
	It completes the proof.
	\end{proof}
\begin{lemma}\label{l2.3}Assume $N>4\alpha,2<q< 2^*, 0<\alpha<1$, $0<\mu <\phi_{\alpha,N}(\frac{N-4\alpha}{2})$.	Let $\{v_n\}\subset H^{\alpha}(\R^N)$ be a Palais-Smale  sequence of $I$ at
	level $d_1$ and $v_n\rightharpoonup 0\text{ in } H^{\alpha}(\R^N)\text{
		as }n\rightarrow \infty$. If there exists a sequence $\{r_n\}\subset
	\R^+, \text{ with } r_n\rightarrow 0$ as $n\rightarrow \infty$ such
	that $\bar v_n(x):=r_n^{\frac{N-2\alpha}{2}}v_n (r_n x)$ converges weakly
	in $\dot{H}^\alpha(\R^N)$ and almost everywhere to some $v_0\in
	\dot{H}^\alpha(\R^N)\text{ as }n\rightarrow \infty$ with $v_0\not\equiv 0$, then $v_0$ solves
	(\ref{1.12}), the sequence
	$z_n(x):=v_n(x)-v_0(\frac{x}{r_n})r_n^{\frac{2\alpha-N}{2}}\rightharpoonup 0 \text{ in } \dot{H}^\alpha (\R^N)$ and $z_n(x)$ is a
	Palais-Smale sequence of $I$ at level $d_1-I_\mu(v_0)$. \end{lemma}
\begin{proof}
	First, we prove that $v_0$ solves (\ref{1.12}) and
	$I(z_n)=I(v_n)-I_\mu(v_0)$. Fix a ball $B(0,r)$ and a test
	function $\phi\in C^\infty_0(B(0,r))$.
	Since
	\begin{equation}\label{2.2j}
	\bar v_n\rightharpoonup v_0 \text{  in } \dot{H}^\alpha(\R^N),
	\end{equation}
 it implies
	\begin{equation}\label{j62.21}
	\begin{split}
	&\ \ \ \langle\phi, I'_\mu( v_0)\rangle+o(1)
	\\&	=\langle\phi, I'_\mu(\bar v_n)\rangle\\&=c_{N,\alpha}\int_{\R^N}\int_{\R^N}\frac{(\bar v_n(x)-\bar v_n(y))(\phi(x)-\phi(y))}{|x-y|^{N+2\alpha}}dxdy-\mu\int_{\R^N}\frac{\bar v_n(x)\phi(x)}{|x|^{2\alpha}}dx
	-\int_{\R^N}\bigl(\bar v^+_n(x)\bigl)^{{2^*}-1}\phi(x)
	dx\\
	&=c_{N,\alpha}\int_{\R^N}\int_{\R^N}\frac{(\bar v_n(x)-\bar v_n(y))(\phi(x)-\phi(y))}{|x-y|^{N+2\alpha}}
	dxdy-\mu\int_{\R^N}\frac{\bar v_n(x)\phi(x)}{|x|^{2\alpha}}dx-\int_{\R^N}\bigl(\bar v^+_n(x)\bigl)^{{2^*}-1}\phi(x)
	dx\\
	& \ \ \ +r_n^{2\alpha}\int_{\R^N} a({r_n x})\phi(x)\bar v_n(x)
	dx-r_n^{N-\frac{N-2\alpha}{2}q}\int_{\R^N}k(r_n x)\phi(x) (\bar v^+_n(x))^{q-1}dx+o(1)\\
	&=c_{N,\alpha}\int_{\R^N}\int_{\R^N}\frac{(v_n(x)-v_n(y))(\phi_n(x)-\phi_n(y))}{|x-y|^{N+2\alpha}}
	dxdy-\mu\int_{\R^N}\frac{ v_n(x)\phi_n(x)}{|x|^{2\alpha}}dx-\int_{\R^N}(v^+_n(x))^{{2^*}-1}\phi_n(x)
	dx\\
	& \ \ \ +\int_{\R^N} a(x)\phi_n(x) v_n(x)dx -\int_{\R^N}
	k(x)\phi_n(x)(v^+_n(x))^{q-1} dx+o(1) =o(1)\text{ as }n\rightarrow
	\infty,
	\end{split}
	\end{equation}
	where $\phi_n=r_n^{-\frac{N-2\alpha}{2}}\phi(\frac{x}{r_n})$.
	The last equality in  (\ref{j62.21})  holds since
	$$\int_{\R^N}|\phi_n(x)|^2dx=r_n^{2\alpha}\int_{\R^N}|\phi(x)|^2dx=o(1),$$
	$$\|\phi\|_{\dot{H}^\alpha(\R^N)}=\|\phi_n\|_{\dot{H}^{\alpha}(\R^N)}=\|\phi_n\|_{{H}^{\alpha}(\R^N)}+o(1),\text{ as
	}n\rightarrow \infty.$$
	Thus $v_0$ is a nontrival critical point of $I_\mu$.  Noting the fact $ N>4\alpha$, $\mu <\phi_{\alpha,N}(\frac{N-4\alpha}{2})$ and $\phi_{\alpha,N}$ is a strictly increasing, it follows 
	$$\begin{aligned}
	\eta_\mu>\frac{2\alpha}{N-2\alpha},\,\, (1+\eta_\mu)\frac{(N-2\alpha)p}{2}\geq (1+\eta_\mu)(N-2\alpha)>N,\,\, \forall \,\, p\geq 2
	\end{aligned}$$
	then by Lemma \ref{l4.6} and (\ref{1.17}), it follows
	\begin{equation}\label{2.10}
	\int_{\R^N}|v_0(x)|^pdx\leq c,\,\, \forall \,\, p\geq 2
	\end{equation}
	which implies that $v_0\in L^2(\R^N)$.
	Let
	$$ z_n(x)=v_n(x)-r_n^{\frac{2\alpha-N}{2}}v_0(\frac{x}{r_n})\in
	H^{\alpha}(\R^N).$$
	
	From (\ref{2.10}) and $v_0\in L^p(\R^N)$ for all $p\in [2,2^*)$, it follows
	\begin{equation}\label{x1}
	\begin{split}
	\int_{\R^N}|v_0(\frac{x-x_n}{r_n})r_n^{\frac{2\alpha-N}{2}}|^pdx=r_n^{N-p\frac{(N-2\alpha)}{2}}\|v_0\|^p_{L^p(\R^N)}\rightarrow 0 ,\,\,\text{ as }n\rightarrow\infty,\text{ for all }{2\leq p< {2^*}},
	\end{split}
	\end{equation}
	Thus $z_n\rightharpoonup 0 \text{
		in }H^{\alpha}(\R^N)\text{ as }n\rightarrow \infty$. Now we prove
	that $\{z_n\}$ is a Palais-Smale  sequence of $I$ at level $d_1-I_\mu
	(v_0)$.
	By the Br\'{e}zis-Lieb Lemma and the weak convergence, similar to Lemma \ref{q} in the Appendix, we can prove that
	$$I(z_n)=I(v_n)-I_\mu(v_0),$$
	$$ \langle I'(z_n),\phi\rangle=o(1)$$ as $n\rightarrow \infty$.
	It completes the proof.
\end{proof}

 {\bf Proof of Theorem \ref{1.1}.}
 By Lemma \ref{3.5} in the appendix, we can assume that $\{u_n\}$ is
bounded. Up to a subsequence, $\text{ let }n\rightarrow \infty$, we
assume that
\begin{gather}
\label{t1.11}u_n\rightharpoonup u \text{ in } H^{\alpha}(\R^N),\\
\label{t1.12}u_n\rightarrow u \text{ in } L^p_{\mathrm{loc}}(\R^N)\,\text{ for }2\leq
p<{2^*},\\
\label{t1.13}u_n\rightarrow u \text{ a.e. in } \R^N.
\end{gather}
Denote $v_n(x)=u_n(x)-u(x)$, then $\{v_n\}$ is
a Palais-Smale  sequence of $I$ and $v_n\rightharpoonup0$  in
$H^{\alpha}(\R^N)$ and
\begin{gather}
\label{t2.11}v_n\rightharpoonup 0 \text{ in } H^{\alpha}(\R^N),\\
\label{t2.12}v_n\rightarrow 0 \text{ in } L^p_{\mathrm{loc}}(\R^N)\,\text{ for }2\leq
p<{2^*},\\
\label{t2.13}v_n\rightarrow 0 \text{ a.e. in } \R^N.
\end{gather}
 Then by Lemma \ref{q} we know that
\begin{gather}
\label{4.1}I(v_n)=I(u_n)-I(u)+o(1),\text{ as }n\rightarrow \infty,\\
\label{4.2}I'(v_n)=o(1),\text{ as }n\rightarrow \infty,\\
\label{4.3}\|v_n\|_{H^{\alpha}(\R^N)}=\|u_n\|_{H^{\alpha}(\R^N)}-\|u\|_{H^{\alpha}(\R^N)}+o(1),\text{
as }n\rightarrow \infty.
\end{gather}

Without loss of generality, we may assume that
$$\|v_n\|^2_{H^{\alpha}(\R^N)}\rightarrow l>0\text{ as
}n\rightarrow\infty.$$ In fact if $l=0$, Theorem 1.1 is proved for $l_1=0,l_2=0,l_3=0$.

{\bf Step 1}: getting rid of the blowing up bubbles caused
by unbounded domains.

Suppose there exists a constant $0<\delta <\infty$ such that
\begin{equation}\label{4.5}
(\int_{\R^N}(v_n^+(x))^{q}dx)^{\frac{2}{q}}\geq \delta >0,\text{ for } 2<q<2^*.
\end{equation}
By  interpolation inequality, it follows
\[\|v_n\|_{L^q}\leq \|v_n\|_{L^2}^\lambda \|v_n\|^{1-\lambda}_{L^{2^*}}, \text{ for } 2<q<2^*\]
where $0<\lambda<1$. Thus there exists a $\tilde{\delta}>0$ such that
\[\|v_n\|_{L^2}^2\geq \tilde{\delta}>0.\]
By Lemma \ref{3.1}, there exists a subsequence still denoted by $\{v_n\}$, such that one
of the following two cases occurs.

i) Vanish occurs.
$$\displaystyle\forall \,0< R<\infty, \sup_{y\in\R^N}\int_{B(y,R)}|v_n(x)|^2dx\rightarrow 0\text{ as }n\rightarrow \infty.$$
By Lemma \ref{3.2}, (\ref{4.10}) and Sobolev inequality,  it follows $$\int_{\R^N}(v_n^+(x))^{q}dx\rightarrow 0\text{ as
}n\rightarrow \infty,\,\,\forall\,2< q<2^*,$$ which contradicts
(\ref{4.5}).

ii) Nonvanish occurs.

There exist $\beta
>0,\,0<\bar R<\infty,\,\{y_n\}\subset\R^N \text{ such that }$
\begin{equation}\label{4.4}
\liminf_{n\rightarrow \infty}\int_{y_n+B_{\bar R}}
|v_n(x)|^2dx\geq \beta
>0.
\end{equation}

We claim  that $|y_n|\rightarrow\infty$ as
$n\rightarrow\infty$. Otherwise, if there exists a constant $M>0$ such that $|y_n|\leq M$, then we can choose a $R_2>0$ large enough such that
\begin{equation}\label{2.30}
\int_{y_n+B_{\bar R}}
|v_n(x)|^2dx\leq \|v_n\|_{L^2 (B(0,R_2))}^2\rightarrow 0 \text{ as } n \rightarrow \infty,
\end{equation}
which contradicts (\ref{4.4}).

To proceed, we first construct the Palais-Smale sequences of $I^\infty$.

Denote $\bar
v_n(x)=v_n(x+y_n)$.
Since $\|\bar v_n\|_{H^{\alpha}(\R^N)}=\|v_n\|_{H^{\alpha}(\R^N)}\leq c$, without
loss of generality, we assume that $\text{ as } n\rightarrow
\infty$,
\begin{equation}\label{2.46j}
\begin{aligned}
\bar v_n\rightharpoonup v_0&\text{ in }H^{\alpha}(\R^N),\\
\bar v_n\rightarrow v_0& \text{ in } L^p_{\mathrm{loc}}(\R^N),\text{ for any }
1< p < {2^*}.
\end{aligned}
\end{equation}
Then $\forall \phi\in C^\infty_0(B(0,r))\text{ as }
n\rightarrow \infty$,
\begin{equation}\label{2.43j}
\begin{aligned}\int_{\R^N}\frac{
	\bar v^+_n(x) \phi(x)}{|x+y_n|^{2\alpha}}dx&=\int_{B(0,r)}
\frac{  \bar v^+_n(x)\phi(x)}{|x+y_n|^{2\alpha}}dx
\\& \leq\frac{2}{|y_n|^{2\alpha}}\int_{\R^N}|\bar v_n(x)\phi|dx+o(1)
\\&\leq \frac{2}{|y_n|^{2\alpha}}\|\phi\|_{H^\alpha}\|\bar v_n\|_{H^\alpha}+o(1)
\\&\leq \frac{c}{|y_n|^{2\alpha}}+o(1)=o(1).
\end{aligned}
\end{equation}
Similarly we have \begin{equation}\label{4.17}\int_{\R^N}\frac{ (\bar v^+_n(x))^2}{|x+y_n|^{2\alpha}}dx=o(1)\text{ as }
n\rightarrow \infty.
\end{equation}
Since $v_n\rightharpoonup 0 \text{ in
}H^{\alpha}(\R^N)$ and $\dis\lim_{n\rightarrow\infty}a(x+y_n)=\bar a$,
we have as $n\rightarrow\infty$,
$$
o(1)= \int_{\R^N}a(x)v_n(x)\phi_n(x) dx=\int_{\R^N}\bar a \bar v_n(x)\phi(x)
dx+\int_{\R^N}[a(x+y_n)-\bar a] \bar v_n(x)\phi(x) dx $$
and $$|\int_{\R^N}[a(x+y_n)-\bar a] \bar v_n(x)\phi(x) dx|\leq c(\int_{\R^N}|a(x+y_n)-\bar
a
|^2\phi(x)^2dx)^{1/2}=o(1),$$
that is,
\begin{equation}\label{s3}
\int_{\R^N}\bar a \bar v_n(x)\phi(x) dx=o(1)=\int_{\R^N}a(x)v_n(x)\phi_n(x)
dx\text{ as } n\rightarrow \infty.
\end{equation}
Similarly we have
\begin{equation}\label{s4}
\int_{\R^N} k(x) (v_n^+(x))^{q-1} \phi_n(x) dx=\int_{\R^N}\bar k (\bar
v^+_n(x))^{q-1}\phi(x) dx=o(1)\text{ as } n\rightarrow \infty.
\end{equation}
Recall that $v_n$ is a Palais-Smale
sequence of $I$, by (\ref{2.46j}) and (\ref{4.17})-(\ref{s4}) we have
\begin{equation}o(1)=\langle I'(v_n),\phi_n\rangle=\langle {I^\infty}' (\bar
v_n),\phi\rangle+o(1)=\langle {I^\infty}' (
v_0),\phi\rangle+o(1),\text{ as
}n\rightarrow \infty.\end{equation}
This shows
that 
$v_0$ is a weak solution of
(\ref{(2.1)}).


We claim that $v_0\not\equiv 0$.
From (\ref{4.5}), we may assume that there exists a sequence $\{y_n\}$
satisfying (\ref{4.4}) and
\begin{equation}\label{z3}\int_{B(y_n,R)}(v^+_n(x))^qdx=b+o(1)>0,\text{ as
}n\rightarrow \infty,
\end{equation} where $b>0$ is a constant.
If $v_0\equiv 0$, we have $$\int_{B(0,R)}(\bar v^+_n(x))^qdx=\int_{B(y_n,R)}
(v^+_n(x))^qdx=o(1)\text{ as } n\rightarrow \infty\text{ for }0<R<\infty$$ which
contradicts (\ref{z3}).


Denote $z_n(x)= v_n(x)-v_0(x-y_n)$. Since \begin{align*}
I(v_n)&=\frac{1}{2}\int_{\R^N}\Bigl (|(-\Delta)^{\alpha/2}v_n(x)|^2+a(x)|v_n(x)|^2-\mu\frac{|v_n(x)|^2}{|x|^{2\alpha}}
\Bigl )dx
\\&\ \ -\frac{1}{{2^*}}\int_{\R^N}(v_n^+(x))^{{2^*}}dx-\frac{1}{q}\int_{\R^N}k(x)(v_n^+(x))^{q}
dx\\
&=\frac{1}{2}\int_{\R^N} \Bigl (|(-\Delta)^{\alpha/2}\bar v_n(x)|^2+a(x+y_n)|\bar
v_n(x)|^2-\mu\frac{|\bar v_n(x)|^2}{|x+y_n|^{2\alpha}} \Bigl )dx-\frac{1}{{2^*}}\int_{\R^N}(\bar
	v_n^+(x))^{{2^*}}dx\\
&\ \ \ -\frac{1}{q}\int_{\R^N}k(x+y_n)(\bar v_n^+(x))^q dx\\
&=\frac{1}{2}\int_{\R^N}\Bigl (|(-\Delta)^{\alpha/2}\bar v_n(x)|^2+\bar a|\bar
v_n(x)|^2 \Bigl )dx-\frac{1}{q}\int_{\R^N}\bar k(\bar v_n^+(x))^q dx-\frac{1}{{2^*}}\int_{\R^N}(\bar
v_n^+(x))^{{2^*}}dx+o(1)\\
&=I^\infty(\bar v_n)+o(1),
\end{align*}
where the last equality but one is a result of (\ref{4.17}), therefore, as $n\rightarrow \infty$,
\begin{gather}\label{4.20}\|z_n\|_{H^{\alpha}(\R^N)}=\|\bar
v_n\|_{H^{\alpha}(\R^N)}-\|v_0\|_{H^{\alpha}(\R^N)}+o(1)=\|
v_n\|_{H^{\alpha}(\R^N)}-\|v_0\|_{H^{\alpha}(\R^N)}+o(1), \\
\label{4.21} I(z_n)=I^\infty(\bar v_n)-I^\infty(v_0)+o(1)=I(v_n)-I^\infty
(v_0)+o(1).
\end{gather}
Hence
$z_n\rightharpoonup0\text{ in }H^{\alpha}(\R^N)\text{ as } n\rightarrow
\infty$,
and $z_n$ is a Palais-Smale sequence of
$I$. From (\ref{4.10}) in Lemma \ref{l4.5}, it follows 
$ \|v_0^-\|_{H^\alpha}=0$, that is $v_0\geq 0$ a.e. in $\R^N$.
Then by Brezis-Lieb Lemma and (\ref{4.10}), there exists a constant $c>0$ such that
\begin{equation}\label{2.56j}
\begin{aligned}
\int_{\R^N}(z_n^+(x))^qdx
= \int_{\R^N}(v_n^+(x))^qdx-\int_{\R^N}(v_0^+(x))^qdx+o(1)\leq  \int_{\R^N}(v_n^+(x))^qdx-c
\end{aligned}\end{equation}
where   the last inequality follows from the fact $v_0\not\equiv 0$.
If $\|z_n\|_{L^q(\R^N)}\rightarrow \delta_2>0\text{ as
}n\rightarrow\infty$, from (\ref{2.56j})
and the boundedness of $\|v_n\|_{L^q}$, then one can repeat Step 1 for finite times ($l_1$ times).
Thus we  obtain a new Palais-Smale sequence of $I$, without loss of generality still denoted by $v_n$, such that
\begin{equation}\label{2.45}
d=I(u)+I(v_n)+\sum^{l_1}_{k=1}I^\infty(u_k)+o(1),
\end{equation}
\begin{equation}\label{2.46}
v_n(x)=u_n(x)-u(x)-\sum^{l_1}_{k=1}u_k(x-y_n^k),\text{ with } y_n^k\rightarrow \infty,
\end{equation}
\begin{equation}\label{2.42}
\|v_n^+\|_{L^q(\R^N)}\rightarrow 0
\end{equation}
as $n\rightarrow\infty$.

{\bf Step 2:}\, Getting rid of the blowing up bubbles caused by the  critical terms.

Suppose there exists $0<\delta<\infty$ such that
 \begin{equation}\label{4.6}
\inf_{n\in \N}\int_{\R^N}\bigl (v_n^+(x)\bigl )^{2^*}dx\geq \delta>0,\,\,\text{ for
some } \,0<R<\infty.
\end{equation}

It follows from Lemma \ref{l6} that there exist  two sequences $\{r_n\}\subset \R^+$ and $\{x_n\}\subset \R^N$,
such that \begin{equation}\label{2.25}\bar v_n(x)\rightharpoonup v_0\not\equiv 0\text{ in } \dot{H}^{\alpha}(\R^N),\end{equation}
	where \begin{equation}\label{2.3jj}
	\begin{aligned}
	\bar v_n(x)=\begin{cases}
	r_n^{\frac{N-2\alpha}{2}}v_n(r_n x) &\text{ when }\frac{x_n}{r_n} \text{ is bounded, }\\
	r_n^{\frac{N-2\alpha}{2}}v_n(r_n x+x_n) &\text { when } |\frac{x_n}{r_n}|\rightarrow \infty.
	\end{cases}
	\end{aligned}
	\end{equation}

Now we claim that $r_n\rightarrow 0 \text{ as } 	n\rightarrow \infty.$
In fact there exists a $R_1>0$ such that
\begin{equation}\label{jt3.8}
	\int_{B(0,R_1)}|v_0(x)|^pdx=\delta_1>0, \text{ for }2\leq p<2^*.
\end{equation}
From the Sobolev compact embedding, (\ref{t1.12}), (\ref{2.25}) and (\ref{jt3.8}), we have that for all $r>0$,
\[ v_n\rightarrow 0\text{ in } L^p(B(0,r)) \text{ for all }2\leq p<2^*,\]
\[ \bar v_n\rightarrow v_0
\text{ in } L^p(B(0,r)) \text{ for all }2\leq p<2^*,\]
\begin{equation}\label{t1.18}
\begin{aligned}
0&\neq\|v_0\|_{L^2(B(0,R_1))}^2+o(1)\\&=\int_{B(0,R_1)}|\bar v_n(x)|^2dx\\&=\begin{cases}
r_n^{-2\alpha}\int_{B(0,r_nR_1)}|v_n(x)|^2dx, &\text{ if }\frac{x_n}{r_n} \text{ is bounded, }\\
r_n^{-2\alpha}\int_{B(x_n,r_nR_1)}|v_n(x)|^2dx &\text { if } |\frac{x_n}{r_n}|\rightarrow \infty.
\end{cases}
\end{aligned}
\end{equation}

If $|\frac{x_n}{r_n}|\rightarrow \infty$, then there exists a constant $\bar c $ such that
\begin{equation}\label{2.47}
\begin{aligned}
\ \ \ 0<\bar c& <r_n^{-2\alpha}\int_{B(x_n,r_nR_1)}|v_n(x)|^2dx \\&\leq c r_n^{-2\alpha}\bigl(\int_{B(x_n,r_nR_1)}|v_n(x)|^qdx\bigl)^{2/q}
(w_N(r_nR_1)^N)^{1-\frac{2}{q}}\\
&\leq c r_n^{N(1-\frac{2}{q})-2\alpha}\bigl(\int_{\R^N}|v_n(x)|^qdx\bigl)^{2/q}
\end{aligned}
\end{equation}
Then from (\ref{2.42}) (\ref{2.47}) and the fact $q<2^*$,   it follows that $r_n\rightarrow 0$. Similarly, if $\frac{x_n}{r_n}$ is bounded, we also have that $r_n\rightarrow 0$.



For the case that $\frac{x_n}{r_n}$ is bounded and $\bar v_n(x)=
	r_n^{\frac{N-2\alpha}{2}}v_n(r_n x)$, define $z_n(x) =v_n(x)-v_0(\frac{x}{r_n})r_n^{\frac{2\alpha-N}{2}}$. It follows from Lemma
 \ref{3.7} that $\{z_n\}$ is a Palais-Smale sequence of $I$
 satisfying
 \begin{eqnarray}
 I(z_n)=I(v_n)-I_\mu(v_0)+o(1), \text { as } n\rightarrow\infty. 
 \end{eqnarray}
 and   $ z_n \rightharpoonup 0 $ in $H^\alpha (\R^N)$. 
Since $v_0$ satisfies (\ref{1.12}), from Lemma \ref{l4.6}, (\ref{1.5j})  and (\ref{1.11j}) there exists $\va_1>0$ such that
 \begin{equation}\label{2.25j}
 v_0(x)=\va_1^{\frac{2\alpha-N}{2}}U_\mu(\frac{x}{\va_1}), I_\mu(v_0)=D_\mu.
 \end{equation}
 Let $R_n^1=r_n\va_1$, from (\ref{2.25j}), it follows
 \begin{equation}
r_n^{\frac{2\alpha-N}{2}} v_0(\frac{x}{r_n})
={(R^1_n)}^{\frac{2\alpha-N}{2}}U_\mu(\frac{x}{R_n^1})=U_\mu^{R^1_n}(x),
 \end{equation}
 with $R^1_n\rightarrow 0$. Then from (\ref{4.1})  it follows
 \begin{equation}
 \begin{aligned}
 z_n(x)=v_n(x)-U_\mu^{R^1_n}(x),\\ I(z_n)=I(v_n)-D_\mu+o(1)
 \end{aligned}\end{equation}
  with $R^1_n\rightarrow 0$.
From (\ref{4.10}),  we have that $z_n\geq 0, \text{ a.e.  in } \R^N$. From Lemma \ref{l4.8}, let $a=v_n,b=U_\mu^{R^1_n}$, it follows
 \begin{equation}\label{3.18}
 \begin{aligned}
 &\int_{\R^N}\bigl (z_n^+(x)\bigl )^{2^*}dx=\int_{\Omega}\bigl (z_n(x)\bigl )^{2^*}dx\\ \leq &\int_{\Omega}\bigl (v_n(x)\bigl )^{2^*}-(U_\mu^{R^1_n}(x))^{2^*}dx\\=&\int_{\Omega}\bigl (v_n^+(x)\bigl )^{2^*}dx-C
\\ \leq & \int_{\R^N}\bigl (v_n^+(x)\bigl )^{2^*}dx-C
 \end{aligned}
 \end{equation}
 where $\Omega=\{x|z_n(x)\geq 0\}\bigcap \R^N$.

 For the case that $|\frac{x_n}{r_n}|\rightarrow \infty$ and $\bar v_n(x)=
 r_n^{\frac{N-2\alpha}{2}}v_n(r_n x+x_n)$, define $z_n(x) =v_n(x)-v_0(\frac{x-x_n}{r_n})r_n^{\frac{2\alpha-N}{2}}$. It follows from Lemma
 \ref{l2.3} that $\{z_n\}$ is a Palais-Smale sequence of $I$
 satisfying
 \begin{eqnarray}
 \label{4.12}I(z_n)=I(v_n)-I_0(v_0)+o(1), \text { as } n\rightarrow\infty. 
 \end{eqnarray}
 and $ z_n \rightharpoonup 0 $ in $H^\alpha (\R^N)$.
 Since $v_0$ satisfies (\ref{1.4}), from Lemma \ref{l4.6}, (\ref{1.5j})  and (\ref{1.10j}) there exists $\va_1>0$ such that
 \begin{equation}\label{2.2jj}
 v_0(x)=\va_1^{\frac{2\alpha-N}{2}}U_0(\frac{x-y}{\va_1}), I_0(v_0)=D_0.
 \end{equation}
 Let $\bar R_n^1=r_n\va_1$ and $ x_n^1=y+\va_1x_n$, from (\ref{2.2jj}), it follows
 \begin{equation}\label{2.26j}
 r_n^{\frac{2\alpha-N}{2}} v_0(\frac{x}{r_n})
 ={(\bar R^1_n)}^{\frac{2\alpha-N}{2}}U_0(\frac{x-x^1_n}{\bar R_n^1})=U_0^{\bar R^1_n,x_n^1}(x),
 \end{equation}
 with $\bar R^1_n\rightarrow 0$. Then from (\ref{4.1})  it follows
 \begin{equation}
 \begin{aligned}
 z_n(x)=v_n(x)- U_0^{\bar R^1_n,  x_n^1}(x),\\ I(z_n)=I(v_n)-D_0+o(1)=I(u_n)
 \end{aligned}\end{equation}
 with $\bar R^1_n\rightarrow 0$.
 Similar to (\ref{3.18}), it follows
 \begin{equation}
 \begin{aligned}
 \int_{\R^N}\bigl (z_n^+(x)\bigl )^{2^*}dx \leq  \int_{\R^N}\bigl (v_n^+(x)\bigl )^{2^*}dx-C
 \end{aligned}
 \end{equation}

If still there exists a $\bar \delta>0, \text{ such that }$
\[\dis\int_{\R^N}\bigl (z_n^+(x)\bigl )^{2^*}dx\geq \bar \delta
 >0,\] then repeat the previous argument.
 From (\ref{3.18}) and the fact \[\int_{\R^N}\bigl (z_n^+(x)\bigl )^{2^*}dx\leq \|v
 _n\|_{H^\alpha}^{2^*}\leq c,\] we deduce that the iteration must stop after
 finite times.
 That is to see, there exist nonnegative constants $l_2,l_3$ and  a new Palais-Smale sequence of $I$,
 (without loss of generality) denoted by $\{v_n\}$, such that as $ n\rightarrow \infty$,
 \begin{equation}\label{2.30j}
 \begin{aligned}
 d&=I(v_n)+I(u)+\sum^{l_1}_{k=1}I^\infty(u_k)+l_2 D_\mu+l_3D_0,\,\,\,\\
 v_n(x)&=u_n(x)-u(x)-\sum^{l_1}_{k=1}u_k-\sum^{l_2}_{j=1}U_\mu^{R^j_n}(x)+\sum^{l_3}_{i=1} U_0^{\bar R^i_n,  x_n^i}(x)
 \end{aligned}
 \end{equation}
 with $ R^j_n\rightarrow 0$, $\bar R^i_n\rightarrow 0$ and $| x_n^i/\bar R_n^i|\rightarrow \infty$.
 \begin{equation}\label{4.15}\int_{\R^N}(v_n^+)^{2^*} dx=o(1)
 ,\,\,\,\|v_n\|_{{L^q}(\R^N)}\rightarrow 0 
 \end{equation}
 and
 \begin{equation}\label{4.15jj}v_n\rightharpoonup 0\text{ in } H^{\alpha}(\R^N).
 \end{equation}

Then from the fact $<I'(v_n),v_n>=o(1)$, it follows
\begin{equation}\label{2.58}
\begin{aligned}\|v_n\|^2_{H^\alpha(\R^N)}&\leq c\int_{\R^N}(| \bigl (-\Delta )^{\alpha/2}v_n(x)|^2+a(x)|v_n(x)|^2-\mu\frac{(v_n^+(x))^2}{|x|^{2\alpha}}\bigl )dx\\&=c\bigl(\int_{\R^N}k(x)(v_n^+(x))^{q}dx+ \int_{\R^N}\bigl (v_n^+(x)\bigl )^{2^*}dx\bigl ) \rightarrow 0
\end{aligned}
\end{equation} as $n\rightarrow \infty$. From (\ref{2.58}), it gives that
\begin{equation}\label{2.59}
I(v_n)=o(1).
\end{equation}
From (\ref{2.30j})-(\ref{2.59}), the proof of Theorem 1.1 is complete.

\section{Proof of Theorem 1.2}
Now we are ready to prove Theorem 1.2 by Mountain Pass Theorem \cite{BN} and
Theorem 1.1.

 {\bf Proof of Theorem \ref{b}:}
From
\begin{align*}
I (tu)=&\frac{t^2}{2}\Bigl[\int_{\R^N}|(-\Delta)^{\alpha/2}
u(x)|^2dx-\mu\int_{\R^N}\frac{(u^+(x))^2}{|x|^{2\alpha}}dx+\int_{\R^N}a(x)|u(x)|^2
dx\Bigl]\\
&-\frac{|t|^{{2^*}}}{{2^*}}\int_{\R^N}{(u^+(x))^{{2^*}}}dx-\frac{
|t|^q}{q}\int_{\R^N}k(x)(u^+(x))^q dx,
\end{align*}
we deduce that for a fixed $u\not\equiv 0$ in $H^{\alpha}(\R^N)$, $I
(tu)\rightarrow -\infty$ if $t\rightarrow \infty$. Since
 $$
\int_{\R^N}(u^+(x))^q dx\leq
C\|u\|^q_{H^{\alpha}({\R^N})},\,\,\,\int_{\R^N}{(u^+(x))^{{2^*}}} dx\leq
C\|u\|^{{2^*}}_{H^{\alpha}({\R^N})},
 $$
 we have
$$
I(u)\geq
c\|u\|_{H^{\alpha}(\R^N)}^2-C(\|u\|_{H^{\alpha}(\R^N)}^q+\|u\|_{H^{\alpha}(\R^N)}^{{2^*}}).
$$
Hence, there exists $r_0>0$  small such that $I(u)\Bigl|_{\pa
B(0,r_0)}\geq\rho>0$ for $q,\,{2^*}>2$.

As a consequence,  $I(u)$ satisfies the geometry structure of
Mountain-Pass Theorem. Now define
$$ c^*=:\inf_{\gamma \in \Gamma} \sup_{t
\in [0,1]}I(\gamma (t)),$$ where $ \Gamma=\{\gamma \in
C([0,1],H^{\alpha}(\R^N)):\gamma(0)=0,\gamma(1)=\psi_0\in H^{\alpha}(\R^N)\}$ with $I(t\psi_0)\leq 0$ for all $t\geq 1$.

To complete the proof of Theorem 1.2, we need to verify that
$I(u)$ satisfies the local Palais-Smale conditions. According to
Remarks 1),
 we only need to verify that
\begin{equation}\label{4.28}
c^*<\min\{\frac{\alpha}{N}S_{\alpha,\mu}^{\frac{N}{2\alpha}},\frac{\alpha}{N}S_{\alpha}^{\frac{N}{2\alpha}},
J^\infty\}=\min\{\frac{\alpha}{N}S_{\alpha,\mu}^{\frac{N}{2\alpha}},
J^\infty\}.
\end{equation}  

 Set $\dis v_\va(x)=\frac{U_\mu^\va(x)}{(\int_{\R^N}{|U_\mu^\va(x)|^{{2^*}}}dx)^{1/{2^*}}}$. We
 claim
\begin{equation}\label{4.29}
\max_{t>0}I(tv_\va)<\frac{\alpha}{N}S_{\alpha,\mu}^{\frac{N}{2\alpha}}.
\end{equation}  In
fact,  from (\ref{1.5}) it is easy to calculate  the following estimates
\begin{eqnarray}
\|v_\va\|_{\dot{H}^\alpha(\R^N)}^2=S_{\alpha,\mu}, \label{jt3.19}
\end{eqnarray}
\begin{eqnarray}
\dis
\int_{\R^N}(v_\va(x))^2dx= c\va^{2\alpha}\|U_\mu\|_{L^2}^2=
O(\va^{2\alpha}), \text{ for }N> 4\alpha, \mu <\phi_{\alpha,N}(\frac{N-4\alpha}{2}),
 \label{jt3.20}
\end{eqnarray}
\begin{eqnarray}
\int_{\R^N}(v_\va(x))^qdx= O(\va^{\frac{(2\alpha-N)q}{2}+N}).\label{jt3.21}
\end{eqnarray}
 Since $2^*>q>2$ we
 have
 \begin{equation}\label{l}
  O(\va^{2\alpha})=o(\va^{\frac{(2\alpha-N)q}{2}+N}).\,\,
 \end{equation}
 Denote
$t_\va$ the attaining point of $\dis\max_{t>0}I(tv_\va)$, similar to the proof of Lemma 3.5 in \cite{DG} we can
prove that $t_\va$ is uniformly bounded.
In fact, we consider the function
\begin{equation}
\begin{split}
h(t)&=I \left(tv_\va\right)\\&=\frac{t^2}{2}(\|(-\Delta)^{\alpha/2}v_\va\|^2_{L^2(\R^N)}-\mu\int_{\R^N}\frac{v_\va^2}{|x|^{2\alpha}}dx+\int_{\R^N}a(x) (v_\va(x))^2 dx)\\& \ \ \ \ -\frac{t^{{2^*}}}{{2^*}}\int_{\R^N}
{(v_\va(x))^{{2^*}}}dx-\frac{t^{{q}}}{{q}}
\int_{\R^N}(k(x)v_\va(x))^qdx
\\& \geq
\frac{ct^2}{2}\|v_\va\|^2_{H^\alpha(\R^N)}-\frac{Ct^{{2^*}}}{{2^*}}\|v_\va\|_{H^\alpha(\R^N)}^{2^*}-\frac{Ct^{{q}}}{{q}}\|v_\va\|_{H^\alpha(\R^N)}^q.
\end{split}
\end{equation}
Since $\displaystyle\lim_{t\rightarrow +\infty} h(t)=-\infty$ and $h(t)>0$ when $t$ is closed to $0$, then $\displaystyle\max_{t>0}h(t)$ is attained for $t_\va>0$.  From the fact $\dis\int_{\R^N}{(v_\va(x))^{{2^*}}}dx=1$,
it follows\begin{equation}\label{3.11j}
\begin{aligned}
h'(t_\va)&=t_\va(\|(-\Delta)^{\alpha/2}v_\va\|^2_{L^2(\R^N)}-\mu\int_{\R^N}\frac{v_\va^2}{|x|^{2\alpha}}dx+\int_{\R^N}a(x) (v_\va(x))^2 dx)\\&\ \ \ -{t_\va^{{2^*-1}}} -t_\va^{q-1}\int_{\R^N}k(x)(v_\va(x))^qdx=0.\end{aligned}\end{equation}
Since $k(x)>0$, from (\ref{jt3.19}) and (\ref{jt3.20}) for $\va$ sufficiently small,
we have \begin{equation}\label{3.12j} t_\va^{{2^*-2}}\leq {\|(-\Delta)^\alpha v_\va\|^2_{L^2(\R^N)}-\mu\int_{\R^N}\frac{v_\va^2}{|x|^{2\alpha}}dx+\int_{\R^N} a(x)(v_\va(x))^{2} dx}<2S_{\alpha,\mu}.\end{equation}
Then \begin{equation}\label{3.10}
\begin{aligned}
&\|(-\Delta)^{\alpha/2}v_\va\|^2_{L^2(\R^N)}-\mu\int_{\R^N}\frac{v_\va^2}{|x|^{2\alpha}}dx+\int_{\R^N}a(x) (v_\va(x))^2 dx\\&= t_\va^{2^*-2}+t_\va^{q-2}\int_{\R^N}k(x)(v_\va(x))^qdx\\&
\leq t_\va^{2^*-2}+(2S_{\alpha,\mu})^{\frac{q-2}{2^*-2}}
\int_{\R^N}k(x)(v_\va(x))^qdx.
\end{aligned}
\end{equation}
Choosing $\va>0$ small enough, by (\ref{jt3.19})-(\ref{jt3.21}), there exists a constant $\alpha_1>0$ such that $t_\va>\alpha_1>0$.
 Combining this with (\ref{3.12j}), it implies that $t_\va$ is bounded for $\va>0$ small enough.

 Hence, for
$\va>0$  small,
\begin{align*}
 \max_{t>0}I(tv_\va)&=I (t_\va v_\va)\\&\leq
 \max_{t>0}\Bigl\{\frac{t^2}{2}\int_{\R^N}(|(-\Delta)^{\alpha/2}
v_\va(x)|^2-\mu\frac{v_\va^2}{|x|^{2\alpha}})dx-\frac{t^{{2^*}}}{{2^*}}\int_{\R^N}
(v_\va(x))^{{2^*}}dx\Bigl\}\\&\ \ \  -O(\va^{\frac{(2\alpha-N)q}{2}+N}) +
O(\va^{2\alpha}), \\
&<\frac{\alpha}{N}S_{\alpha,\mu}^{\frac{N}{2\alpha}}\text{ \,\,( by (\ref{l}) )}.
\end{align*}
This completes the proof of (\ref{4.29}). By the definition of
$c^*$, we have $c^*<\frac{\alpha}{N}S_{\alpha,\mu}^{\frac{N}{2\alpha}}$.

Next we verify
\begin{equation}\label{4.30}
c^*<J^\infty.
\end{equation}
Let $\{u_0\}$ be the minimizer of
$J^\infty,\,I^\infty(u_0)=J^\infty$ and
\[\int_{\R^N}\Bigl(|(-\Delta)^{\alpha/2} u_0(x)|^2+\bar a
|u_0(x)|^2
\Bigl)dx=\int_{\R^N}\bar
k(u_0^+(x))^q dx+\int_{\R^N}(u_0^+(x))^{2^*} dx.\]
 Let
$$\begin{aligned}g(t)&=I^\infty(tu_0)=\frac{1}{2}t^2\int_{\R^N}\Bigl(|(-\Delta)^{\alpha/2} u_0(x)|^2+\bar a
|u_0(x)|^2
\Bigl)dx\\&\ \ \ -\frac{t^q}{q}\int_{\R^N}\bar
k(u_0^+(x))^q dx-\frac{t^{2^*}}{2^*}\int_{\R^N}(u_0^+(x))^{2^*} dx,\end{aligned}$$
$$
g'(t)=t\int_{\R^N}\Bigl(|(-\Delta)^{\alpha/2} u_0(x)|^2+\bar a
|u_0(x)|^2
\Bigl)dx-t^{q-1}\int_{\R^N}\bar
k(u_0^+(x))^q dx-t^{2^*-1}\int_{\R^N}(u_0^+(x))^{2^*}dx.$$ Thus $g'(t)\geq0$ if $t\in(0,1)$; $g'(t)\leq 0$ if
$t\geq 1$. Then
\begin{equation}\label{4.31}
\begin{aligned}
g(1)&=I^\infty(u_0)=\max_{l}I^\infty (u);\\
\text{ where }&l=\{tu_0,t\geq 0, u_0\text{ fixed }\}.
\end{aligned}
\end{equation}
Since there exists a $t_0>0$ such that $\displaystyle \sup_{t\geq
0}I(tu_0)=I(t_0u_0)$, from (\ref{4.31}) and the assumptions of
$a(x) \text{ and } k(x)$, we have
$$\displaystyle J^\infty=I^\infty (u_0)\geq
I^\infty (t_0u_0)>I(t_0 u_0)=\sup_{t\geq 0}I(t u_0).$$ It
  proves (\ref{4.30}). By (\ref{4.29}) and (\ref{4.30}) we have
(\ref{4.28}). Then the proof is completed.

\section{Appendix}
In this appendix, we give some lemmas and detailed proofs for the convenience of the reader.
\begin{lemma}\label{3.1}(Lemma 2.1, \cite{ZC}) 
	Let $\{\rho_n\}_{n\geq 1}$ be a sequence in $L^1(\R^N)$ satisfying
	\begin{equation}\label{(2.4)}\rho_n\geq 0\,\,\, \text{on }
	\R^N,\,\,\,\lim_{n\rightarrow\infty}\int_{\R^N}\rho_n(x)dx=\la>0,
	\end{equation} where $\la>0$ is fixed. Then there exists a
	subsequence $\{\rho_{n_k}\}$ satisfying one of the  following two
	possibilities:
	
	(1) (Vanishing):
	\begin{equation}\label{2.5}
	\lim_{k\rightarrow\infty} \sup_{y\in\R^N}\int_{B(y,R)}\rho_{n_k}(x)
	dx=0,\,\,\,\, \text{for all } R<+\infty.
	\end{equation}
	
	(ii) (Nonvanishing): $\exists \alpha>0, R<+\infty$ and
	$\{y_k\}\subset \R^N$ such that $$\displaystyle\liminf_{k\rightarrow
		+\infty}\int_{y_k+B_R}\rho_{n_k}(x)dx\geq \alpha >0.$$
\end{lemma}

\begin{lemma}\label{3.2}(Lemma 2.2, \cite{PTJ})
	If $\{u_n\}$ is bounded in $H^\alpha(\R^N) $ and for some $R>0$, we have
	$$\sup_{y\in\R^N}\int_{B(y,R)}|u_n(x)|^2 dx\rightarrow 0
	\,\,\text{as}\,\,n\rightarrow \infty.$$
	Then $u_n\rightarrow 0$ in $L^q (\R^N)$, for $2<q<\frac{2N}{N-2\alpha}$.
\end{lemma}

\begin{lemma}\label{3.5}
	Let $\{u_n\}$ be a Palais-Smale  sequence of $I$ at level
	$d\in\mathbb{R}$. Then $d\geq 0$ and $\{u_n\}$ $\subset$ $H^{\alpha}
	(\R^N)$ is bounded. Moreover,
	every Palais-Smale sequence for $I$ at a level zero converges
	strongly to zero.
\end{lemma}
\begin{proof}
	Since $a(x)\geq 0$, $\bar a>0$ and
	$\dis\inf_{x\in \R^N} a(x)=\hat{a}>0$, we have
	$$\|u_n\|_{\dot{H}^\alpha (\R^N)}^2+\int_{\R^N}
	a(x)|u_n(x)|^2dx\geq
	c\|u_n\|_{H^{\alpha}(\R^N)}^2,$$ and hence
	for $2<q<2^*$
	\begin{eqnarray}\label{l3.1}
	\begin{split}
	d+1+o(\|u_n\|) \geq &
	I(u_n)-\frac{1}{q}\langle I'(u_n), u_n\rangle\\
	=&(\frac{1}{2}-\frac{1}{q})\int_{\R^N}\bigl (|(-\Delta)^{\alpha/2}u_n(x)|^2-\mu\frac{|u_n(x)|^{2}}{|x|^{2\alpha}}+a(x)|u_n(x)|^2\bigl )dx \\&+(\frac{1}{q}-\frac{1}{{2^*}})\int_{\R^N}(u_n^+(x))^{{2^*}} dx\\
	\geq & C\|u_n\|_{H^{\alpha}(\R^N)}^2.
	\end{split}
	\end{eqnarray}
	It follows  that $\{u_n\}$ is bounded in $H^{\alpha}(\R^N)$ for $2<q<2^*$. Since
	$$d=\lim_{n\rightarrow
		\infty}I(u_n)-\frac{1}{q}\langle I'(u_n),u_n\rangle\geq
	C\limsup_{n\rightarrow \infty}\|u_n\|_{H^{\alpha}(\R^N)}^2,$$ we have
	$d\geq 0$. Suppose now that $d=0$, we obtain from the above inequality
	that
	$$\lim_{n\rightarrow \infty}\|u_n\|_{H^{\alpha}(\R^N)} =0.$$
\end{proof}

\begin{lemma}\label{l4.5}
Let $\{u_n\}$ be a  Palais-Smale sequence of  $I$ at level
$d\in\mathbb{R}$ and $u^+_n=\max \{u_n,0\}$.
 Then $\{u_n^+\}$ is also a Palais-Smale sequence of $I$ at level $d$.

\end{lemma}
\begin{proof} By the definition of $I$ we have that as $n\rightarrow \infty$
	$$\begin{aligned}I(u_n)=&\frac{1}{2}\int_{\R^N}\Bigl (|(-\Delta)^{\alpha/2}u_n(x)|^2-\mu\frac{|u_n(x)|^{2}}{|x|^{2\alpha}}
	+a(x)|u_n(x)|^2\Bigl )
	dx\\&-\dis\frac{1}{2^*}\int_{\R^N}{(u_n^+(x))^{2^*}}dx-\frac{1}{q}\int_{\R^N}k(x)(u_n^+(x))^q
	dx\rightarrow d,\end{aligned}$$
	and
	\begin{equation}
	\begin{aligned}
	&\ \ <I'(u_n),\phi>\\&=C_{N,\alpha}\int_{\R^N}\int_{\R^N}\frac{(u_n(x)-u_n(y))(\phi(x)-\phi(y))}{|x-y|^{N+2\alpha}}dxdy+\int_{\R^N}a(x)u_n(x)\phi(x) dx-\mu\int_{\R^N}\frac{u_n(x)\phi(x)}{|x|^{2\alpha}}dx\\&\ \ -\int_{\R^N}(u^+_n(x))^{2^*-1}\phi(x)dx-\int_{\R^N}k(x)(u_n^+(x))^{q-1}\phi(x)dx\rightarrow 0, \text{ for all }\phi \in H^\alpha(\R^N).
	\end{aligned}
	\end{equation}
	Taking $\phi=-u_n^-=\min \{u_n,0\}$, from the fact
	\begin{equation}\label{4.8}
u_n(x)=u_n^+(x)-u_n^-(x), \,\, u_n^+(x)u_n^-(x)=0,
	\end{equation}
	 we have
	\begin{equation}\label{4.9}
	\begin{aligned}
	o(1)&=<I'(u_n),-u_n^->\\&=-C_{N,\alpha}\int_{\R^N}\int_{\R^N}\frac{(u_n(x)-u_n(y))(u_n^-(x)-u_n^-(y))}{|x-y|^{N+2\alpha}}dxdy+\mu\int_{\R^N}\frac{u_n(x)u^-_n(x)}{|x|^{2\alpha}}dx\\&\ \ -\int_{\R^N}a(x)u_n(x)u_n^-(x) dx +\int_{\R^N}(u^+_n(x))^{2^*-1}u_n^-(x)dx+\int_{\R^N}k(x)(u_n^+(x))^{q-1}u_n^-(x)dx\\
	&=C_{N,\alpha}\int_{\R^N}\int_{\R^N}\frac{(u_n^-(x)-u_n^-(y))^2}{|x-y|^{N+2\alpha}}dxdy-\mu\int_{\R^N}\frac{(u^-_n(x))^2}{|x|^{2\alpha}}dx
	\\&\ \ \ +C_{N,\alpha}\int_{\R^N}\int_{\R^N}\frac{u_n^+(x)u_n^-(y)+u_n^+(y)u_n^-(x)}{|x-y|^{N+2\alpha}}dxdy +\int_{\R^N}a(x)(u_n^-(x))^2 dx \\&\geq c \|u_n^-\|_{H^\alpha}^2,
	\end{aligned}
	\end{equation}
	from (\ref{4.9}) and the fact $u^+_n(x)\geq 0,\, u^-_n(x)\geq 0,\, a(x)>0$, then
	\begin{equation}\label{4.10}
	\|u_n^-\|_{H^\alpha}\rightarrow 0,
	\end{equation}
	and
	\begin{equation}\label{4.11}
	\begin{aligned}
	\int_{\R^N}\int_{\R^N}\frac{2(u^+_n(x)-u^+_n(y))(u^-_n(x)-u^-_n(y))}{|x-y|^{N+2\alpha}}dxdy\rightarrow 0.
	\end{aligned}
	\end{equation}	Then
	from (\ref{4.8}) and (\ref{4.10})-(\ref{4.11}), we have
	\begin{equation}\label{4.12j}
	\begin{aligned}
&	 \int_{\R^N}\int_{\R^N}\frac{(u_n(x)-u_n(y))^2}{|x-y|^{N+2\alpha}}dxdy\\=&\int_{\R^N}\int_{\R^N}\frac{(u^+_n(x)-u^+_n(y))^2+(u^-_n(x)-u^-_n(y))^2-2(u^+_n(x)-u^+_n(y))(u^-_n(x)-u^-_n(y))}{|x-y|^{N+2\alpha}}dxdy\\=&\int_{\R^N}\int_{\R^N}\frac{(u_n^+(x)-u^+_n(y))^2}{|x-y|^{N+2\alpha}}dxdy+o(1).
	\end{aligned}
	\end{equation}
	That is
	\begin{equation}\label{4.13}
	\|u_n\|_{\dot{H}^\alpha}=\|u_n^+\|_{\dot{H}^\alpha}+o(1).
	\end{equation}
	Thus
	\[\lim_{n\rightarrow \infty}I(u_n^+)=\lim_{n\rightarrow \infty}I(u_n)=d\]
and \[I'(u_n^+,\phi)=I'(u_n,\phi)\rightarrow 0\]
as $n\rightarrow \infty$.  This complete the proof.
\end{proof}
\begin{lemma}\label{l4.6}
All nontrivial critical points of $I_\mu$  are the positive solutions of (\ref{1.4}).	
\end{lemma}
\begin{proof}
	Let $u\not\equiv 0$ and $ u\in H^\alpha(\R^N)$ be a nontrivial critical point of $I_\mu$.
	First, arguing as in the proof of Lemma \ref{l4.5} (similar to (\ref{4.9}) and (\ref{4.10})), we can obtain that $\|u^-\|_{H^\alpha}=0$ which gives that
	\begin{equation}\label{4.14}
	u\geq 0 \text{ a.e. in } \R^N.
	\end{equation}
	Then  for any $x_0\in \R^N$,
	\begin{equation}\label{4.15j}
	(-\Delta)^\alpha u=\mu\frac{u}{|x|^{2\alpha}}+|u|^{2^*-2}u\geq 0,\text{ a.e. in } B(x_0,1), \,\int_{\R^N}\frac{|u(x)|}{1+|x|^{N+2\alpha}}dx\leq c\|u\|_{L^2}\leq c,
	\end{equation}
	 from Proposition 2.2.6	in \cite{SL}, we have
	 $u$ is lower semicontinous in $B(x_0,1)$. Combining this with (\ref{4.14}), it follows $u(x_0)\geq 0$. Then $u(x)\geq 0$ pointwise in $\R^N$.

	Next we claim that $u>0$ in $\R^N$. Otherwise there exist $x_1\in \R^N$ such that $u(x_1)=0$. Since 
	 $u$ is lower semicontinuous in $\overline {B(x_1,1/2)}$,
	 from Proposition 2.2.8 in \cite{SL}, it follows $u\equiv 0 \text{ in } \R^N$. This contradicts the assumption $u$ is nontrivial.
\end{proof}	
	
Let $\{u_n\}$ be a Palais-Smale sequence at level $d$. Up to a subsequence, we
assume that
$$u_n\rightharpoonup u \,\,\,\text{ in } H^{\alpha}(\R^N)\text{ as
}n\rightarrow \infty.$$ Obviously, we have $I'(u)=0$.
Let $v_n(x)=u_n(x)-u(x)$, as $n\rightarrow \infty$,
\begin{eqnarray}
v_n\rightharpoonup 0\text{ in
}H^{\alpha}(\R^N),\ \ \ \ \ \ \ \  \label{jl2.7}\\
v_n\rightarrow
0\text{ in
}L^{q}_{\mathrm{loc}}(\R^N)\text{ for all } 2< q<2^*,\\
v_n\rightarrow 0, \text{a.e. in } \R^N.\label{jl4.19}
\end{eqnarray}
As a consequence, we have the following Lemma.
\begin{lemma}\label{q}
	$\{v_n\}$ is a Palais-Smale  sequence for $I$ at level
	$d_0=d- I(u)$. \end{lemma}
\begin{proof}

	For $\phi(x)\in C^\infty_0(\R^N)$, there exists a $B(0,r)$ such that $\dis \mathrm{supp}\phi \subset B(0,r)$. Then$\text{ as } n\rightarrow \infty$,
	\begin{equation}
	\begin{aligned}\label{l5.16}
	|\int_{\R^N}k(x)(v_n^+(x))^{q-1} \phi(x) dx|\leq c
	|\int_{B(0,r)}(v_n^+(x))^{q-1} \phi(x)dx|=o(1),
	\end{aligned}
	\end{equation}
	\begin{equation}\label{l5.17}
	\begin{aligned}
	|\int_{\R^N }\frac{v_n^+(x) \phi(x)}{|x|^{2\alpha}}dx|
&	= |\int_{|x|\leq r }\frac{v_n^+(x)^{} \phi(x)}{|x|^{2\alpha}}dx|\\
&
	\leq c \int_{|x|\leq r }\frac{v_n^+(x)^{}}{|x|^{2\alpha}}dx
	\\&=(\int_{|x|\leq r }|v_n|^{\bar q}dx)^{\frac{1}{\bar q}}(\int_{|x|\leq r }\frac{1}{|x|^{2\alpha\frac{\bar q}{\bar q-1}}}dx)^{1-\frac{1}{\bar q}}=o(1).
	\end{aligned}
	\end{equation}
	where $\frac{N}{N-2\alpha}<\bar q<2^*$.
	
	By (\ref{jl2.7}), (\ref{l5.16}) and (\ref{l5.17}), we have $\langle \phi,
	I'(v_n)\rangle=o(1)\text{ as }n\rightarrow
	\infty$.  Then similar to (\ref{4.10}), we have  \begin{equation}
	\label{4.22j}\|v_n^-\|_{\dot{H}^\alpha}\rightarrow 0,\|u^-\|_{\dot{H}^\alpha}=0.
	\end{equation}
	By Sobolev inequality, (\ref{4.10}) and (\ref{4.22j}) it follows
	\[\|u_n\|_{L^q}=\|u_n^+\|_{L^q}+o(1),\,\|v_n\|_{L^q}=\|v_n^+\|_{L^q}+o(1), \|u\|_{L^q}=\|u^+\|_{L^q}.\]
	Then by the Br\'{e}zis-Lieb Lemma in \cite{BN}  as
	$n\rightarrow\infty$, we have
	\begin{equation}\label{4.22}\int_{\R^N}( v_n^+(x))^{q}dx=\int_{\R^N}
	(u_n^+(x))^{q}dx-\int_{\R^N}(u^+(x))^{q}dx+o(1)\,\, \text{for all }\  2\leq
	q\leq {2^*} . \end{equation}
	Similarly
	\begin{equation}\label{4.23}\int_{\R^N }\bigl (z_n^+(x)\bigl )^{2^*}dx=\int_{\R^N}
	\bigl(u^+_n(x)\bigl)^{2^*}dx-\int_{\R^N} (u^+(x))^{2^*}dx+o(1). \,
	\end{equation}
	\begin{eqnarray}\label{4.24}
	\begin{split}
	&\int_{\R^N}\int_{ \R^N}\frac{|u_{n}(x)-u_{n}(y)|^2}{|x-y|^{N+2\alpha}}dxdy\\=&	\int_{\R^N}\int_{ \R^N}\frac{|(v_{n}(x)+u(x))-(v_{n}(y)+u(y))|^2}{|x-y|^{N+2\alpha}}dxdy\\
	=&\int_{\R^N}\int_{ \R^N}\frac{|v_{n}(x)-v_{n}(y)|^2+|u(x)-u(y)|^2+2(v_n(x)-v_n(y))(u(x)-u(y))}{|x-y|^{N+2\alpha}}dxdy\\
	=&\int_{\R^N}\int_{ \R^N}\frac{|v_{n}(x)-v_{n}(y)|^2}{|x-y|^{N+2\alpha}}dxdy+\int_{\R^N}\int_{ \R^N}\frac{|u(x)-u(y)|^2}{|x-y|^{N+2\alpha}}dxdy+o(1).
	\end{split}
	\end{eqnarray}
	Hence from (\ref{4.22})-(\ref{4.24}), it follows
	$I(v_n)=I(u_n)-I(u)+o(1)=d-I(u)+o(1)$.
\end{proof}

\begin{lemma}\label{l4.8}Assume $t\geq b> 0\text{ and }q>1$, then
	 \[ t^q-(t-b)^q\geq b^q.\]
\end{lemma}
\begin{proof}
	Let $f(t)=  t^q-(t-b)^q$, it follows
	\[f'(t)=qt^{q-1}-q(t-b)^{q-1}>0\text{ for } t\geq b> 0,\,q>1.\]
	Then $f(t)= t^q-(t-b)^q\geq f(b)=b^q.$
\end{proof}

\bigskip

\noindent{\bf Acknowledgement} \ \ The research was supported by the Natural Science Foundation of China
	(11271141) and the China Scholarship Council (201508440330).

	\author{ Lingyu Jin \\\small{Department of Mathematics, South China Agricultural
			University,}
		\\\small{ Guangzhou 510642, P. R. China}
		\\ \small {Email: jinlingyu300@126.com}

\end{document}

